\pdfoutput=1
\documentclass[11pt]{amsart}

\usepackage[
text={450pt,575pt},
headheight=9pt,
centering
]{geometry}

\usepackage[english]{babel}
\usepackage[utf8]{inputenc}
\usepackage{paralist,url,verbatim}
\usepackage{amscd}
\usepackage[T1]{fontenc}
\usepackage{xspace}
\usepackage{comment}
\usepackage{amsmath,amsthm,amssymb,amsfonts}
\usepackage{mathtools}
\usepackage{enumitem}
\usepackage{wasysym}
\usepackage{color}
\usepackage[bookmarksopen=true]{hyperref}
\definecolor{darkblue}{RGB}{0,0,160}
\hypersetup{
colorlinks,%
citecolor=black,%
filecolor=black,%
linkcolor=darkblue,%
urlcolor=darkblue
}

\theoremstyle{plain}
\newtheorem{theorem}{Theorem}[section]
\newtheorem*{mainTheorem}{Theorem \ref{t:generalThmFirstStatement}}

\newtheorem{question}[theorem]{Question}

\newtheorem{corollary}[theorem]{Corollary}
\newtheorem*{mainCorollary}{Corollary \ref{c:regularityArbitrary}}
\newtheorem{lemma}[theorem]{Lemma}
\newtheorem{proposition}[theorem]{Proposition}

\theoremstyle{definition}
\newtheorem{remark}[theorem]{Remark}
\newtheorem{definition}[theorem]{Definition}

\newtheorem{example}[theorem]{Example}

\renewcommand\>{\rangle}
\newcommand\<{\langle}

\newcommand{\sta}{\operatorname{star} }
\newcommand{\reg}{\operatorname{reg} }
\newcommand{\vcd}{\operatorname{vcd} }
\newcommand{\ima}{\operatorname{im} }
\newcommand{\pdim}{\operatorname{projdim}}
\newcommand{\rank}{\operatorname{rank} }
\newcommand{\chara}{\operatorname{char} }

\newcommand{\Th}{\operatorname{Th} }

\renewcommand{\P}{\mathcal{P}}
\newcommand{\Ww}{\mathcal{W}}

\newcommand{\N}{\mathbb{N}}
\newcommand{\Z}{\mathbb{Z}}
\newcommand{\Q}{\mathbb{Q}}
\newcommand{\FF}{\mathbb{F}}

\let\para\S
\renewcommand{\S}{\mathcal{S}}
\newcommand{\C}{\mathcal{C}}
\newcommand{\NN}{\mathcal{N}}
\newcommand{\ZZ}{\mathbb{Z}}

\newcommand{\kk}{\Bbbk}

\newcommand{\bfC}{\mathbf{C}}

\newcommand{\pp}{\mathfrak{p}}
\newcommand{\mm}{\mathfrak{m}}
\newcommand{\D}{\Delta}
\newcommand{\lk}{\operatorname{lk}}
\newcommand{\cd}{\operatorname{cd}}
\newcommand{\ind}{\operatorname{index}}
\newcommand{\Tor}{\operatorname{Tor}}
\newcommand{\GL}{\operatorname{GL}}
\newcommand{\Cay}{\operatorname{Cay}}
\newcommand{\Ker}{\operatorname{Ker}}

\newcommand{\minel}{\hat{0}}

\newcommand{\swi}{{\'S}wi{\k{a}}tkowski\xspace}

\newcommand{\ritter}{\mathrm{S}}
\newcommand{\face}[1]{2^{(#1)}}

\newcommand{\excise}[1]{}

\begin{document}

\title{Linear syzygies, hyperbolic Coxeter groups and regularity}

\author[A.~Constantinescu]{Alexandru Constantinescu}
\address{
Mathematisches Institut,
Freie Universität Berlin,
Arnimallee 3, 14195 Berlin, Germany}
\urladdr{\url{http://userpage.fu-berlin.de/aconstant/Main.html}}

\author[T.~Kahle]{Thomas Kahle}
\address {Fakultät für Mathematik,
OvGU Magdeburg, Universitätsplatz 2, 39106 Magdeburg, Germany}
\urladdr{\url{http://www.thomas-kahle.de}}

\author[M.~Varbaro]{Matteo Varbaro}
\address{Dipartimento di Matematica,
 Universit\`a di Genova,
 Via Dodecaneso 35, Genova 16146, Italy}
\urladdr{\url{http://www.dima.unige.it/~varbaro/}}



\thanks{Large parts of this research were carried out at Mathematisches
For\-schungs\-in\-sti\-tut Oberwolfach within the Research in Pairs
program.  The work was continued with support from the MIUR-DAAD Joint
Mobility Program (Proj. n. 57267452).}

\begin{abstract}
We show that the virtual cohomological dimension of a
Coxeter group is essentially the regularity of the Stanley--Reisner
ring of its nerve.  Using this connection between geometric group theory
and commutative algebra, as well as techniques from the
theory of hyperbolic Coxeter groups, we study the behavior of the
Castelnuovo--Mumford regularity of square-free quadratic monomial
ideals.  We construct examples of such ideals which exhibit
arbitrarily high regularity after linear syzygies for arbitrarily many
steps.  We give a doubly logarithmic bound on the regularity as a
function of the number of variables if these ideals are
Cohen--Macaulay.
\end{abstract}

\maketitle

\section{Introduction}
The \emph{Castelnuovo--Mumford regularity} captures the complexity of
finitely generated graded $R$-modules, where $R = \kk[x_1,\dots,x_n]$
is a standard graded polynomial ring in $n$ variables over a
field~$\kk$.  We focus on the case of modules of the kind $R/I$, where
$I$ is a homogeneous ideal of~$R$.  A fundamental question is how big
the regularity of $R/I$ can be, when $I\subseteq R$ is generated in
fixed degree.  The following are some important results in this area.

\begin{enumerate}[label=\roman*),ref=\roman*]
\item\label{it:mayrmeyer} For any $d\geq 2$, Mayr and Meyer
\cite{mayr82:_compl_word_probl_commut_semig_polyn_ideal} provided
ideals $I\subseteq R$ generated in degrees $\leq d$ for which
$\reg R/I $ is doubly exponential in the number of variables~$n$, as
explained by Bayer and Stillman in~\cite{bayer1988complexity}.
\item Caviglia and Sbarra showed in \cite{caviglia2005characteristic}
that $\reg R/I\leq (2d)^{2^{n-2}}$ provided $I$ is generated in
degrees $\leq d$.
\item Ananyan and Hochster \cite{ananyanStillmanC} proved that, if $I$
is generated by $r$ forms of degrees $\leq d$, then
$\pdim R/I\leq \phi(r,d)$ provided the characteristic of $\kk$ is zero
or larger than $d$ (here $\phi$ is a function not depending on the
number of variables $n$). This solves Stillman's conjecture
\cite{Peeva2009open} in characteristic zero or bigger than~$d$.  By a
result of Caviglia (see for example
\cite[Theorem~2.4]{mccullough2013bounding}), projective dimension can
be equivalently replaced by regularity in the above statement.
\item McCullough and Peeva provided in \cite{mccullough16:_count}
examples of homogeneous prime ideals $\pp\subseteq R$ such that
$\reg R/\pp$ is not bounded by any polynomial function in the
multiplicity. In particular, this shows that the Eisenbud--Goto
conjecture \cite{eisenbud1984linear} is false.
\item Caviglia, Chardin, McCullough, Peeva and the third named author
noticed in \cite{CMPV17} that, if $\kk$ is algebraically closed, there
exists a function $\phi(e)$ bounding $\reg R/\pp$ from above whenever
$\pp$ is a homogeneous prime ideal of multiplicity~$e$.
\end{enumerate}

The Castelnuovo--Mumford regularity of $R/I$ can be read off the
graded Betti numbers $\beta_{ij}$ of $R/I$ as
$\reg R/I=\max\{j-i:\beta_{ij}\neq 0\}$ (see Section~\ref{s:prelim_CA}
for preliminaries on commutative algebra).  The Mayr--Meyer ideals
have the property that $\beta_{2j}\neq 0$ for a certain
$j>d^{2^{n/10}}$.  That is, their eventually high regularity is
visible early in the resolution, indicating a possible connection
between different homological degrees.  Part of the purpose of the
present paper is to investigate the possibilities for such
connections.  Specifically, we study the behavior of the regularity of
free resolutions that stay linear until a certain homological degree.
As an example for questions concerning the limit behavior of
regularity consider the following open problem.

\begin{question}\label{mfoqs}
Is there a family of quadratically generated ideals
$\{I_n\subseteq R=\kk[x_1,\dots,x_n]\}_{n\in\N}$ with linear syzygies
such that
\[
\lim_{n\to \infty}\frac{\reg R/I_n}{n}>0 \quad \text{?}
\]
\end{question}

Following Green and Lazarsfeld \cite[Section~3a]{green1986projective},
we say that, given an integer $p\geq 1$, $R/I$ \emph{satisfies
property~$N_p$} if $\beta_{ij}=0$ for all $1\leq i\leq p$ and
$j\neq i+1$.  So $R/I$ satisfies property $N_1$ if and only if $I$ is
quadratically generated, it satisfies property $N_2$ if and only if
$I$ is quadratically generated and has linear first syzygies, and so on. 
The \emph{Green--Lazarsfeld index} of $R/I$,
denoted by $\ind R/I$, is the largest $p$ such that $R/I$ satisfies
$N_p$ where by convention, $\ind R/I=\infty$ if $I$ has a 2-linear
resolution, and $\ind R/I=0$ if $I$ is not quadratically generated.
If $I$ is a Mayr--Meyer ideal, then its Green--Lazarsfeld index is at
most one.

As a consequence of Eisenbud--Schreyer's construction of pure modules,
a syzygy degree that appears in a free resolution can be unrelated to
all earlier parts of the resolution~\cite{eisenbud2009betti}.  A
construction due to Ullery shows that for any $p,k\in\N$, $k>p+1$,
there even is a homogeneous ideal $I\subseteq R$ such that $R/I$
satisfies $N_p$ and $\beta_{p+1,k}\neq 0$, see~\cite{Ull14}.  These
constructions need a large number of variables in $R$, though, and are
not efficient enough for Question~\ref{mfoqs}.
Due to the flexibility of resolutions of general ideals, it is
interesting to look at more restricted classes.  For example, Koszul
algebras cannot exhibit extremal behaviour as above.  It is known,
however, that for all $p\ge 2$ there exist families of homogeneous
ideals $I_n\subset R$ such that $R/I_n$ is a Koszul algebra satisfying
$N_p$ and
$\lim_{n\to \infty}(\reg R/I_n) /
{\sqrt[\leftroot{-1}\uproot{2}p]{n}}>0$
\cite[Section~6]{avramov2013subadditivity}.

In this paper we are interested in monomial ideals.  Here the
situation is even more rigid as the following result by Dao, Huneke
and Schweig~\cite{dao2013bounds} illustrates.  If $I\subset R$ is a
square-free monomial ideal such that $R/I$ satisfies $N_p$ for some
$p\geq 2$, then
\[
\reg R/I\leq \log_{\frac{p+3}{2}}\left(\frac{n-1}{p}\right) +2.
\]
In particular, a family which gives a positive answer to Question
\ref{mfoqs} cannot consist of monomial ideals.  In
Section~\ref{s:topHomology} we derive a new doubly logarithmic bound
when $R/I$ is Cohen--Macaulay.

The main motivation for the present paper is the following question:
\begin{question}\label{generalquestion}
Fix an integer $p\geq 2$.  Is there a bound $r(p)$ (independent of
$n$) such that $\reg R/I \le r(p)$ for all monomial ideals
$I\subset R$ such that $R/I$ satisfy~$N_p$?
\end{question}

For $p=2$, a negative answer has been given by the authors
in~\cite{CKV15}.  If $R/I$ is Gorenstein, the answer is positive
by~\cite[Theorem~4]{CKV15}.  If $R/I$ is Cohen--Macaulay, then the
answer is unknown~(Question~\ref{q:boundWithCM}).  In this paper we
give a negative answer for arbitrary~$p$ and begin the search for
constructions that realize the negative answer with as few variables
as possible.

These investigations lead us to the consideration of a connection
between square-free monomial ideals and  Coxeter groups.
It starts from the observation that square-free monomial ideals with
property $N_2$ correspond to right-angled hyperbolic Coxeter groups
(see Section~\ref{s:hyperbolicGroups}).  The study of the geometry and
topology of such groups contains many ideas that we feel can be useful
for commutative algebra.  In Section~\ref{s:vcdmeetsreg} we start to
develop this connection,
proving as a cornerstone the following identity of homological invariants (Theorem \ref{t:cd=reg}).
\begin{equation}\label{eq:fundamentalformula}
\vcd W=\max_{\chara \kk}\{\reg \kk[\NN] \},
\end{equation}
where $W$ is a Coxeter group with nerve $\NN(W)$, $\vcd W$ is the
virtual cohomological dimension of~$W$, and $\kk[\NN(W)]$ is the
Stanley--Reisner ring of the simplicial complex~$\NN(W)$. As the
regularity of a Stanley--Reisner ring depends only on the
characteristic of the field, the maximum is taken over all possible
characteristics, choosing one field for each.

We see~\eqref{eq:fundamentalformula} as a general tool to transfer
results from Coxeter group theory to combinatorial commutative algebra
and vice versa.  For example, when $p=2$,
Question~\ref{generalquestion} is equivalent to the following question
of Gromov:
\begin{question}\label{gromovquestion}
Is there a global bound on the virtual cohomological dimension of
hyperbolic right-angled Coxeter groups?
\end{question}
In fact, a right-angled Coxeter group $W$ is hyperbolic if and only if
$\kk[\NN(W)]$ satisfies $N_2$.  As an
immediate consequence of~\eqref{eq:fundamentalformula} and the bound
of \cite{dao2013bounds} we get:

\begin{corollary}
If $W$ is a hyperbolic right-angled Coxeter group with $n$ generators,
then
\[
\vcd W\leq \log_{\frac{5}{2}}\left(\frac{n-1}{2}\right) +2.
\]
\end{corollary}

Gromov's question had already been answered negatively
in~\cite{januszkiewicz2003hyperbolic}. Later on, new examples were
constructed by Osajda in \cite{osajda2010construction}.  Let
$\face{\D}$ denote the face complex of a simplicial complex $\D$
(Definition~\ref{d:facecomplex}).  Exploiting ideas from Osajda's
construction and~\eqref{eq:fundamentalformula}, we prove
\begin{mainTheorem}
Let $I=I_{\D}\subseteq R$ be a square-free quadratic monomial
ideal. If $\chara(\kk)=0$, then there is a positive integer $N$ and a
square-free monomial ideal
$I'=I_{\D'}\subseteq R' = \kk[y_1,\dots,y_N]$ such that:
\begin{enumerate}[label=\textup{\roman*)},ref=\roman*]
\item $\reg R'/I'= \reg R/I+1$;
\item $\ind R'/I'=\ind R/I$;
\item For each vertex $v$ of $\D'$, $\lk_{\D'}v=\face{\D}$.
\end{enumerate}
\end{mainTheorem}

As a corollary, we get a negative answer to Question
\ref{generalquestion}:
\begin{mainCorollary}
For any positive integers $p$ and $r$, there exists a square-free monomial
ideal $I\subseteq R = \kk[x_1,\ldots ,x_{N(p,r)}]$, such that $R/I$
satisfies $N_p$ and $\reg R/I = r$.
\end{mainCorollary}

The proofs of these statements are contained in
Section~\ref{s:InductiveConstruction}.  The crux of the corollary is
that the number of indeterminates $N(p,r)$ depends on the desired $r$
and~$p$.  In Section \ref{sec:bound-numb-vert} we give an explicit
upper bound for the  minimal number of variables in the corollary
(Theorem~\ref{t:boundOnVertices}).

The following Section~\ref{sec:preliminaries} contains some
preliminaries that we hope will be useful to readers not already
initiated in commutative algebra and geometric group theory.
Section~\ref{s:SRHomRemarks} gathers some new homological properties
of Stanley--Reisner rings inspired by the developments in this paper,
but potentially useful beyond.  In Section~\ref{s:topHomology} we
prove a new doubly logarithmic upper bound on the regularity of
Stanley--Reisner rings of complexes with top homology and
property~$N_p$ (Theorem~\ref{t:tophomology}), which yields the same
bound for all Cohen--Macaulay Stanley--Reisner rings with
property~$N_p$ (Corollary~\ref{c:doubleLog}).
Section~\ref{s:vcdmeetsreg} establishes the fundamental
equality~\eqref{eq:fundamentalformula}.  Finally,
Sections~\ref{s:InductiveConstruction} and~\ref{sec:bound-numb-vert}
give Theorem~\ref{t:generalThmFirstStatement} and an upper bound on
the number of variables necessary for arbitrary regularity with
property~$N_p$ (Theorem~\ref{t:boundOnVertices}).

\section{Preliminaries}\label{sec:preliminaries}
As this paper touches upon the somewhat separated topics of geometric
group theory, commutative algebra, and combinatorics, we introduce
some preliminaries first.

\subsection{Cell  complexes}

A \emph{poset} is a partially ordered set $(\P,\leq)$. For every
element $p \in \P$ we define the subposets
$\P_{\leq p} = \{q\in\P: q\leq p\}$ and
$\P_{\geq p} = \{q\in\P: q\ge p\}$.  We do not assume that $\P$ is
finite.
\begin{definition}\label{d:cellComplex}
An \emph{(abstract) convex cell complex} is a poset $\P$ that
satisfies the following two conditions:
\begin{enumerate}[label=\roman*),ref=\roman*]
\item For each $p \in \P$, the subposet $\P_{\leq p}$ is isomorphic to
the poset of faces of some finite convex polytope (including the empty face).
\item For any $p_1,p_2 \in \P$ the poset
$\P_{\leq p_1} \cap \P_{\leq p _2}$ contains a greatest element.
\end{enumerate}
The elements of $\P$ are called \emph{faces}, and the maximal elements
are called \emph{facets}.  If each of the convex polytopes in
condition (i) are simplices (respectively cubes), then $\P$ is an
\emph{abstract simplicial complex} (respectively an \emph{abstract
cubical complex}).
\end{definition}
Conditions (i) and (ii) imply that, if $\P\neq \emptyset$, it has a
unique minimal element $\minel$, and a well defined rank function. The
minimal element corresponds to the empty face, and the rank function
defines the dimension of a face: $\dim(p) = \rank(p)-1$. The
0-dimensional faces are called \emph{vertices} and the 1-dimensional
faces are called \emph{edges}. The \emph{1-skeleton} of a complex $\P$
is the subposet of elements of rank at most 2.  We also interpret
faces as finite sets of vertices:
$F=\{\text{rank 1 elements of } \P_{\le F}\}$.  In this
interpretation, the partial order is inclusion of sets. This way, a
cell complex is a collection of finite subsets of a (possibly
infinite) vertex set. A cell complex is thus a simplicial complex if
the collection is closed under taking subsets. We can always speak of
the cardinality of a face; however, the rank corresponds to the
cardinality of faces only for simplicial complexes. \emph{Nonfaces}
are collections of vertices which do not correspond to any face. These
can also be ordered by inclusion, and minimal nonfaces are well
defined.

A convex cell complex is a \emph{cell} if it has a unique maximal
element $F_0$; the \emph{boundary} of the cell is the poset
$\P\setminus\{F_0\}$.  The subcomplex of $\P$ \emph{induced} by a
nonempty subset $V$ of its vertex set is
$\P|_{V} = \bigcup_{p\in V} \P_{\ge p}\cup \{\minel\}$.  Some authors
use the term \emph{full subcomplex} for our induced subcomplexes.  Not
all subcomplexes are induced (e.g. the boundary of the triangle is not
an induced subcomplex of the triangle, but all edges are induced
subcomplexes).  A cell complex is \emph{locally finite} if
$\P_{\ge p}$ is a finite poset for every $\minel \neq p\in \P$.

\begin{definition}\label{d:Link}
Let $\P$ be an abstract convex cell complex and $F\in\P$ a face.  The
\emph{link $\lk_\P F$ of $F$ in $\P$} is the abstract convex cell
complex $\P_{\ge F}$.
\end{definition}

\begin{remark}
If $F$ is a vertex,
Definition~\ref{d:Link} yields what is commonly known as the spherical
link at the vertex.  Here we prefer a combinatorial definition as we
do not think of our complexes as embedded in a metric space.
\end{remark}

\begin{remark}
If $\P$ is a cubical or a simplicial complex, then every link is a
simplicial complex. If $\P$ is locally finite, then the link is a
finite  complex.
\end{remark}

\begin{example}
The link of each vertex in the 3-dimensional cube is a triangle. The
link of each vertex of the octahedron is a square. 
\end{example}

\subsubsection{Simplicial complexes} It is easy to check that, for
simplicial complexes, all the "usual" definitions agree with the ones
given above.  A simplicial complex $\Delta$ is \emph{flag}, if all the
minimal nonfaces have cardinality two. Equivalently, no induced
subcomplex is the boundary of a simplex.  For any integer $k\ge 3$,
the \emph{$k$-cycle} is the 1-dimensional simplicial complex with
vertex set $\{v_i\}_{i=0,\dots,k-1}$ and edge set
$\{\{v_{i},v_{i+1 (\textup{mod }k)}\}\}_{i=0,\dots,{k-1}}$. The
following property of simplicial complexes is essential to this paper,
as it has interpretations in both commutative algebra and Coxeter
group theory.

\begin{definition}\label{d:k-large}
Let $k\ge 4$ be an integer. A simplicial complex is \emph{$k$-large},
if it is flag and does not have any induced $j$-cycles for $j<k$.
\end{definition}
A cubical or simplicial complex is \emph{locally $k$-large} if all its
vertex links are $k$-large.  In the literature, $5$-largeness is
sometimes referred to as \emph{flag-no-square} or \emph{Siebenmann's
condition}.  We stress here that all $k$-large complexes must be flag,
and that an induced cycle contains no diagonals.

\begin{example}
Let $\D$ be the boundary of the octahedron, i.e.~$\D$ has vertex set
$\{\pm v_i\}_{i=1,2,3}$ and eight 2-dimensional facets:
$\{\pm v_1,\pm v_2, \pm v_3\}$. This complex is flag, because the
minimal nonfaces are $\{+v_i,-v_i\}$, but it is not $5$-large, because
the vertex subset $\{\pm v_1,\pm v_2\}$ induces a 4-cycle. Adding the
edges $\{+v_i,-v_i\}$ for $i=1,2,3$ to $\D$, we obtain a simplicial
complex without induced 4-cycles, but it is not flag.
\end{example}

\begin{definition}\label{d:facecomplex}
Let $\Delta$ be a simplicial complex.  The \emph{face complex}
$\face{\D}$ is the simplicial complex whose vertex set is the set
of nonempty faces of $\Delta$ and where $F_1,\dots,F_s \in \Delta$
form a face of $\face{\D}$ if and only they are all contained in a
single face of~$\Delta$.
\end{definition}
\begin{example}
The face complex of a $d$-simplex is the $(2^{d+1}-2)$-simplex.
\end{example}

\subsection{Coxeter groups}
We use the notation from Davis' book \cite{davis2008geometry}.  A
\emph{Coxeter system} is a pair $(W,S)$ consisting of a finitely
generated group $W$ and a finite set of distinct generators
$S = \{s_{1},\dots,s_{n}\}$, all different from the identity, such
that $W$ is presented as
\[
W = \< s_{1}, \dots, s_{n} : (s_{i}s_{j})^{m_{ij}} = e \>
\]
for $m_{ij} \in \N \cup \{\infty\}$ with $m_{ii} = 1$, and
$m_{ij} \ge 2$ for $i\neq j$.  The case $m_{ij} = \infty$ means no
relation.  If a group has a presentation as above, then it is a
\emph{Coxeter group}, and $S$ is a set of \emph{Coxeter generators}.
The finite Coxeter groups have been classified by
Coxeter~\cite{Coxeter1935}.  The matrix $M = (m_{ij})_{ij}$ is the
\emph{Coxeter matrix} of $(W,S)$.  If $m_{ij}\in\{1,2,\infty\}$,
then the Coxeter group (or Coxeter system) is \emph{right-angled}.
The elements of $S$ are \emph{letters}, and the elements of $W$ are \emph{words}.
 
A \emph{special subgroup} of $W$ is a subgroup $W_T$ generated by a
subset $T\subseteq S$ of the Coxeter generators. In particular, the
trivial subgroup is special.  By
\cite[Theorem~4.1.6]{davis2008geometry}, $(W_T,T)$ is a Coxeter system
for all $T\subset S$.  A subset $T\subseteq S$ is \emph{spherical}
if $W_T$ is finite.  In this case, $W_T$ and the words in it are also
called spherical.  A \emph{spherical coset} is a coset of a spherical
subgroup.  All spherical cosets are finite. Clearly, being spherical
is closed under taking subsets.
  
\begin{definition}\label{d:nerve}
The \emph{nerve $\NN(W,S)$} of a Coxeter system $(W,S)$ is the
simplicial complex consisting of the spherical sets ordered by
inclusion.
\end{definition}

The nerve of a Coxeter system is always a finite simplicial complex,
with the Coxeter generators as vertices. 

\begin{remark}\label{r:flagRight-angledCorrespondence}
There is a one-to-one correspondence between right-angled Coxeter
groups and flag simplicial complexes given as follows.  Every flag
simplicial complex $\D$ is the nerve of a right-angled Coxeter group
$\Ww(\D)$: the off-diagonal entries of the Coxeter matrix of $\Ww(\D)$
are $m_{ij}=2$ whenever $\{i,j\} \in \D$ and $m_{ij}=\infty$
otherwise.  Conversely, if $(W,S)$ is right-angled, and
$T\subseteq S$, such that any two elements are connected by an edge in
the nerve, then $W_T\cong (\Z/2\Z)^{|T|}$.
\end{remark}

\begin{remark}
In a right-angled Coxeter group a word is spherical if and only if it
can be written with letters that commute pairwise. In particular, if
the presentation is reduced (i.e.~no subword is equal to the word),
then each letter appears at most once.
\end{remark}
  
\begin{example}\label{e:rocket}
Not every simplicial complex is the nerve of a Coxeter system.  The
smallest counterexample occurs on five vertices and is given by the
complex with facets $\{123, 145, 245,\allowbreak 345\}$.  This can be
confirmed using the classification of finite Coxeter groups.
\end{example}

\begin{definition}\label{d:DavisComplex}
The \emph{Davis complex} of a Coxeter system $(W,S)$ is the cell
complex $\Sigma(W,S)$ given by the poset of spherical cosets.
\end{definition}

\begin{remark}\label{r:LinkNerve}
The link of any vertex $w$ of $\Sigma(W,S)$ is the poset of spherical
cosets $wW_T$ for all spherical subsets~$T$.  It is thus isomorphic
to the nerve $\NN(W,S)$.
\end{remark}
  
\begin{remark}\label{r:DavisJanuszkiewich}
Davis and Januszkiewicz have discovered a link between
Stanley--Reisner theory and Coxeter groups that is different from the
developments in our paper.  The cohomology ring $H^*(W,\FF_2)$ is
isomorphic to the Stanley--Reisner ring $\FF_2[\N(W,S)]$.  However,
this connection is a characteristic two phenomenon, as otherwise the
product in the cohomology ring need not be commutative.
See~\cite[Theorem~4.11]{davis1991convex}.
\end{remark}

\subsection{Geometric group theory}
Let $\Gamma$ be a simple graph on a (possibly infinite) vertex set
$V$.  Given two vertices $v,w\in V$, a \emph{path} $e$ from $v$ to $w$
is a subset $\{v=v_0,v_1,v_2,\ldots ,v_k=w\} \subseteq V$, such that
$\{v_i,v_{i+1}\}$ is an edge for all $i=0,\ldots ,k-1$. The {\it
length} of a path is $ \ell(e)=k $. The \emph{distance} between $v$ and
$w$ is
\[
d(v,w):=\min\{\ell(e):e \mbox{ is a path from $v$ to $w$}\}.
\]
If $W\subset V$ is a set of vertices, then
$d(v,W) := \min \{d(v,w), w\in W\}$.  A path $e$ from $v$ to $w$ is a
\emph{geodesic path} if $\ell(e)=d(v,w)$. A \emph{geodesic triangle} of
vertices $v_1$, $v_2$, and $v_3$ consists of three geodesic paths
$e_i$ from $v_i$ to $v_{i+1 (\textup{mod }3)}$ for $i=1,2,3$.  For a
real number $\delta \geq 0$, a geodesic triangle $e_1,e_2,e_3$ is {\it
$\delta$-slim} if $d(v,e_i\cup e_j) \leq \delta$ for all $v\in e_k$
and $\{i,j,k\}=\{1,2,3\}$. The graph $\Gamma$ is {\it
$\delta$-hyperbolic} if each geodesic triangle of $\Gamma$ is
$\delta$-slim, and \emph{hyperbolic} if it is $\delta$-hyperbolic for
some $\delta \geq 0$.

\subsubsection{Hyperbolic groups}
\label{s:hyperbolicGroups}
Let $G$ be a group and $S$ a set of distinct generators of $G$, not
containing the identity.  The \emph{Cayley graph} $\Cay(G,S)$ is the
simple graph with vertex set $G$ and edges $\{g,gs\}$ for all $g\in G$
and $s\in S$.  For example, the vertices of the Davis complex
$\Sigma(W,S)$ of a Coxeter system $(W,S)$ are the elements of~$W$ and
the edges are the cosets of the spherical subgroups~$W_{s_i}$.
Therefore the 1-skeleton of $\Sigma(W,S)$ is the Cayley graph
of~$(W,S)$.  Gromov proved that if $\Cay(G,S)$ is hyperbolic for
some finite set of generators $S$ then it is hyperbolic for any
finite set of generators~$S$
\cite[Theorem~12.3.5]{davis2008geometry}.

This justifies the definition of hyperbolic groups in the following
way.
\begin{definition}\label{d:hyperbolicGroup}
A group $G$ is \emph{hyperbolic} if $\Cay(G,S)$ is a hyperbolic graph
for some (equivalently for any) finite set of generators $S$.
\end{definition}

It is easy to check that $\Z^2$ is not hyperbolic. Therefore, if $G$
contains $\Z^2$ as a subgroup, then $G$ cannot be hyperbolic. By work
of Moussong, for a Coxeter group $(W,S)$ this can be reversed:
\[\mbox{ $W$ is hyperbolic $\iff$ $\Z^2\not\subset W$}.\] 
Combining results of Siebenmann~\cite[Lemma~I.6.5]{davis2008geometry}
and Moussong~\cite[Lemma~12.6.2]{davis2008geometry}, if $(W,S)$ is
right-angled then
\begin{equation}\label{hyper=linsyz}
\mbox{ $W$ is hyperbolic $\iff$ $\NN(W,S)$ has no induced 4-cycles}.
\end{equation}

\subsubsection{Cohomological dimension}

The {\it cohomological dimension} of a group $G$ is
\[
\cd G=\sup\{n:H^n(G;M)\neq 0 \ \mbox{ for some $\Z G$-module }M\},
\]
where $H^n(G;M)$ is the $n$-th group cohomology of $G$ with values
in~$M$ (see \cite[Appendix~F]{davis2008geometry} for equivalent
definitions and some properties).  If $G$ has nontrivial torsion, then
$\cd G=+\infty$ (see \cite[Lemma~F.3.1]{davis2008geometry}).
Therefore the notion is not interesting for groups with torsion, but
this can be rectified.
A group $G$ is {\it virtually torsion-free} if it has a finite index
subgroup which is torsion-free.  It follows from a result of
Serre~\cite[Theorem~F.3.4]{davis2008geometry} that, if $\Gamma$ and
$\Gamma'$ are two finite index torsion-free subgroups of~$G$, then
$\cd \Gamma=\cd \Gamma'$. Thus the following notion is well-defined.
\begin{definition}
Let $G$ be a virtually torsion-free group, and $\Gamma$ some
(equivalently any) finite index torsion-free subgroup of~$G$.  The
{\it virtual cohomological dimension} of $G$ is
\[
\vcd G=\cd \Gamma. 
\]
\end{definition}
Each nontrivial Coxeter group has torsion but, admitting a faithful
linear representation (see \cite[Corollary~D.1.2]{davis2008geometry}),
it is virtually torsion-free.  Thus the virtual cohomological
dimension is always well-defined for a Coxeter group. By
\cite[Corollary~8.5.5]{davis2008geometry}, and using
\cite[Lemma~70.1]{munkres1984elements} to avoid geometric
realizations, the $\vcd$ of a Coxeter group $(W,S)$ can be read off
the nerve $\NN(W,S)$, namely:
\begin{equation} \label{eq:vcd} \vcd
W=\max\{i:\widetilde{H}^{i-1}(\NN(W,S)\setminus \sigma;\Z)\neq 0
{\mbox{ for some } \sigma\in \NN(W,S)}\},
\end{equation}
where $\NN(W,S)\setminus \sigma$ is the restriction of $\NN(W,S)$ to
$S\setminus\sigma$, and $\widetilde{H}^{i}$ denotes the reduced
simplicial cohomology modules.

\subsection{Commutative Algebra}
\label{s:prelim_CA}
Let $n$ be a positive integer, $R=\kk[x_1,\ldots ,x_n]$ the polynomial
ring in $n$ variables over a field $\kk$, and $\mm=(x_1,\ldots ,x_n)$
its irrelevant ideal.  Any quotient $R/I$ by some homogeneous ideal
$I\subseteq R$ has a minimal graded free resolution.
\[
0\rightarrow \bigoplus_{j\in\Z}R(-j)^{\beta_{kj}}\rightarrow \cdots
\rightarrow \bigoplus_{j\in\Z}R(-j)^{\beta_{2j}}\rightarrow
\bigoplus_{j\in\Z}R(-j)^{\beta_{1j}}\rightarrow R\rightarrow R/I\rightarrow 0.
\]
The \emph{Betti number} $\beta_{ij}$ is the number of minimal
generators of degree $j$ of the free module in homological degree~$i$
in the resolution.  It is independent of the particular minimal
resolution and can be computed as
$\beta_{i,j}(R/I) = \dim_\kk\Tor_i(R/I,\kk)_j$.
\begin{definition}
The \emph{Castelnuovo--Mumford regularity} of $R/I$ is
\[ \reg(R/I)=\max\{j-i:\beta_{i,j}(R/I)\neq0\}\]
\end{definition}
If $H^i_\mm$ denotes local cohomology with support in
$\mm$,~\cite[Proposition 20.16]{eisenbud95:_commut_algeb} and
Grothendieck duality imply
$\reg(R/I)=\max\{j+i:H_{\mm}^i(R/I)_j\neq 0\}$,
\begin{definition}
For any positive integer $p$, the $\kk$-algebra $R/I$ \emph{satisfies
property~$N_p$} if:
\[
\beta_{i,j}(R/I)=0 \quad \forall \ i=1,\ldots ,p \text{ and } j\neq i+1.
\]
\end{definition}
\begin{definition}\label{d:Stanley-Reisner ring}
Let $\D$ be a finite simplicial complex with vertex
set~$[n]=\{1,\dots,n\}$.  The \emph{Stanley--Reisner ring} of $\D$,
denoted by $\kk[\D]$, is the quotient of $R$ by the square-free
monomial ideal
\[
I_\D = (\prod_{i\in A}x_i : A\subseteq [n] \textup{ and } A\notin \D).
\]
The ideal $I_\D$ is the \emph{Stanley--Reisner ideal} of $\D$.
\end{definition}
There is a one-to-one correspondence between simplicial complexes and
ideals generated by square-free monomials.  From the definition it
follows that a simplicial complex is flag if and only if its
Stanley--Reisner ideal is quadratic.  The $N_p$ property for
Stanley--Reisner rings was characterized combinatorially
in~\cite[Theorem~2.1]{eisenbud2005restricting}:
\begin{theorem}
The Stanley--Reisner ring $\kk[\D]$ satisfies  $N_p$ if and
only if $\D$ is $(p+3)$-large.
\end{theorem}
The Castelnuovo--Mumford regularity of $\kk[\D]$ can be computed from
the reduced singular cohomology of either induced subcomplexes or links of~$\D$.
More precisely, Hochster's formula for graded Betti
numbers~\cite[Corollary~5.12]{miller05:_combin_commut_algeb} gives
\begin{equation}\label{eq:reg1}
\reg \kk[\D]=\max\{i:\widetilde{H}^{i-1}(\D|_A;\kk)\neq 0 \mbox{ for some }A\subseteq [n] \}.
\end{equation}
On the other hand, by Hochster's formula for local
cohomology~\cite[Theorem~13.13]{miller05:_combin_commut_algeb}:
\begin{equation}\label{eq:reg2}
\reg \kk[\D]=\max\{i:\widetilde{H}^{i-1}(\lk_{\D}\sigma;\kk)\neq 0 \mbox{ for some } \sigma \in \D\}.
\end{equation}

\section{Homological remarks on Stanley--Reisner rings}
\label{s:SRHomRemarks}

In this section, $\D$ is a $d$-dimensional simplicial complex on $n$
vertices, and a face of $\D$ is identified with its set of vertices.
We use some standard algebraic topology (see for example \cite[\para
5]{munkres1984elements}).  For $r\leq d$ let $C_r(\D;\kk)$ be the
$\kk$-vector space spanned by the $r$-dimensional faces of~$\D$.  Let
$\partial_r : C_r(\D;\kk) \to C_{r-1}(\D;\kk)$ be the boundary
operator and $Z_r(\D;\kk) = \Ker \partial_r$ the subspace spanned by the cycles
in~$C_r(\D;\kk)$.  Also write $B_r(\D;\kk) = \ima \partial_{r+1}$ for
the subspace spanned by the boundaries in~$C_r(\D;\kk)$.  An $r$-cycle
$C$ is {\it nontrivial} if $C\notin B_r(\D;\kk)$.  A nontrivial
$r$-cycle $C$ is {\it vertex-minimal} if there is no nontrivial
$r$-cycle $C'$ with $V(C')\subsetneq V(C)$, where for
$C=\sum_{i=1}^lc_iF_i\in C_r(\D;\kk)$ we set
$V(C)=\cup_{i=1}^lF_i$.  For every $v\in V(C)$ define
\[
C_v = \sum_{F_i\ni v}c_i(F_i\setminus v) \in C_{r-1}(\lk_{\D}v;\kk).
\]

If $C\in Z_r(\D;\kk)$, then all codimension one faces containing $v$
sum to zero when applying~$\partial$, so
$\partial(\sum_{F_i\ni v}c_iF_i) = C_v$; therefore
$\partial(C_v) = \partial^2(\sum_{F_i\ni v}c_iF_i) = 0$ in
$C_{r-2}(\D;\kk)$ and thus, since the differentials of
$(C_i(\lk_{\D}v;\kk))_i$ are just the restrictions of the
differentials of $(C_i(\D;\kk))_i$, we get
$C_v\in Z_{r-1}(\lk_{\D}v;\kk).$
\begin{lemma}\label{l:cycles}
Let $C=\sum_{i=1}^lc_iF_i$ be a nontrivial $r$-cycle in $\D$ and
$v\in V(C)$ a vertex.
\begin{enumerate}[label=\roman*),ref=\roman*)]
\item\label{it:cycleLemma1} If $r=d$, then $C_v$ is a nontrivial $(d-1)$-cycle in $\lk_{\D}v$.
\item\label{it:cycleLemma2} If $C$ is vertex-minimal, then $C_v$ is a nontrivial $(r-1)$-cycle in $(\lk_{\D}v)|_{V(C)}$.
\end{enumerate} 
\end{lemma}
\begin{proof}
\ref{it:cycleLemma1} is clear.  For \ref{it:cycleLemma2}, it is
harmless to assume that, in the linear order given to the vertices
of~$\D$, $v$ comes first in~$V(C)$.  Assume there exists
$B_v=b_1{G_1}+\ldots +b_s{G_s}$ with $b_i\in\kk$, such that the
${G_i}$ are $r$-dimensional faces in $(\lk_{\Delta} v)|_{V(C)}$ and
$\partial(B_v)=C_v$.  Consider the $\kk$-linear combination of
$(r+1)$-faces in $\Delta$ defined as
$B=b_1(G_1\cup v)+\ldots +b_s(G_s\cup v)$. Then
\[\partial(B)= -\sum_{F_i\ni v}c_iF_i+B_v.\]
So $A=C+\partial(B)$ is a nontrivial $r$-cycle of $\Delta$ (otherwise
$C$ would be trivial). However, $v\notin V(A)\subset V(C)$ -- a
contradiction.
\end{proof}

\begin{proposition}\label{p:reglink}
There exists a vertex $v\in\D$ such that
\[\reg \kk[\lk_{\D}v]\geq \reg \kk[\D]-1.\]
\end{proposition}
\begin{proof}
Let $V'$ be a subset of the vertex set of $\D$ such that
$\Gamma=\D|_{V'}$ has nontrivial $r$th homology with coefficients
in~$\kk$, where $\reg \kk[\D]=r+1$.  Let $C$ be a vertex-minimal
nontrivial $r$-cycle of~$\Gamma$.  By Lemma~\ref{l:cycles}
\ref{it:cycleLemma2}, $C_v$ is a nontrivial $(r-1)$-cycle in
$(\lk_{\Gamma} v)|_{V(C)}$, for all $v\in V(C)$.  Since
$(\lk_{\Gamma} v)|_{V(C)}=(\lk_{\D} v)|_{V(C)}$, the proposition
follows.
\end{proof}

\begin{proposition}\label{p:dubois}
If $I\subseteq R$ is a homogeneous (not necessarily monomial) ideal
such that $\sqrt{I}$ is a square-free monomial ideal, then for any
$i\in\N, j\in \Z$ the map of $\kk$-vector spaces
\[H^i_{\mm}(R/I)_j\to H^i_{\mm}(R/\sqrt{I})_j\]
is surjective. In particular,
$\reg R/\sqrt{I}\leq \reg R/I \mbox{ \ and \ }\pdim R/\sqrt{I}\leq \pdim R/I.$
\end{proposition}
\begin{proof}
Let $A$ denote $R/I$ localized at~$\mm$.  If $\chara \kk>0$, then the
quotient by the nilradical $A_{\mbox{{\tiny red}}} = A/\sqrt{(0)}$ is
$F$-pure. By \cite[Theorem~6.1]{Schwede09}, in characteristic zero,
$A_{\mbox{{\tiny red}}}$ is DuBois.

So in each case, by \cite[Lemma 3.3, Remark 3.4]{MSS16} the map
\[
H^i_{\mm}(R/I)=H^i_{\mm A}(A)\to H^i_{\mm A}(A_{\mbox{{\tiny red}}}) =
H^i_{\mm}(R/\sqrt{I})
\]
is surjective for any $i\in\N$. Since the above map is homogeneous we
conclude.
\end{proof}

Given two polynomial rings $R=\kk[x_1,\ldots ,x_n]$ and
$R'=\kk[y_1,\ldots ,y_m]$, a map of $\kk$-algebras $f:R\to R'$ is a
{\it monomial map} if $f(x_i)$ is a monomial in $\{y_1,\ldots ,y_m\}$
for all $i=1,\ldots ,n$.

\begin{lemma}\label{l:monmaps}
If $f:R\to R'$ is a monomial map and $I\subseteq R$ is a monomial
ideal, then
\[
\pdim R'/\sqrt{f(I)R'} \leq \pdim R/I.
\]
\end{lemma}
\begin{proof}
Since $\sqrt{I}$ is a square-free monomial ideal,
$\pdim R/I\geq \pdim R/\sqrt{I}$ by Proposition~\ref{p:dubois}.  By a
classical result of Lyubeznik (see the main theorem of
\cite{lyubeznik1984local}) the projective dimension of $R/\sqrt{I}$
equals $\cd(R,I)$, the cohomological dimension of~$I$.  Since the
computation of local cohomology is independent of the base
ring~\cite[Theorem 4.2.1]{BS13}, $\cd(R',f(I)R')\leq \cd(R,I)$.  Again
using \cite{lyubeznik1984local},
$\cd(R',f(I)R')=\pdim R'/\sqrt{f(I)R'}$.
\end{proof}

\begin{proposition}\label{p:regFaceEqual}
If $\face{\D}$ is the face complex of~$\D$, then
$\reg \kk[\D]=\reg \kk[\face{\D}]$.
\end{proposition}
\begin{proof}
Clearly $\reg \kk[\D]\leq \reg \kk[\face{\D}]$ by
Definition~\ref{d:facecomplex} and~\eqref{eq:reg1}. Let $\Gamma$ and
$\Gamma'$ be the Alexander duals of, respectively, $\D$
and~$\face{\D}$. Then, by the Eagon--Reiner
theorem~\cite[Theorem~5.63]{miller05:_combin_commut_algeb},
$\pdim \kk[\Gamma]-1=\reg \kk[\D]$ and
$\pdim \kk[\Gamma']-1=\reg \kk[\face{\D}]$.  It can be checked that
$I_{\Gamma}$ is the following ideal of $R=\kk[x_1,\ldots ,x_n]$:
\[I_{\Gamma}= (\prod_{i\in [n]\setminus \tau}x_i:\tau\textup{~is a
facet of~}\D).\]
For $I_{\Gamma'}\subseteq R'=\kk[y_{\sigma}:\sigma\in\D]$ it holds
that
\[I_{\Gamma'}=(\prod_{\substack{\sigma\in\D \\ \sigma \not\subseteq
\tau}}y_{\sigma}:\tau \textup{~is a facet of~}\D).\]
The map $R\xrightarrow{f} R'$ defined by
$x_i\mapsto \displaystyle{\prod_{\substack{\sigma\in\D,
i\in\sigma}}}y_{\sigma}$ gives $I_{\Gamma'}=\sqrt{f(I_{\Gamma})R'}$,
so that the result follows from Lemma~\ref{l:monmaps}.
\end{proof}

\section{Regularity from top homology}
\label{s:topHomology}
The main result of this section is an improvement of the
\cite{dao2013bounds} bound in the case that $\Delta$ is a
Cohen--Macaulay complex.  In this case, a doubly logarithmic bound for
the regularity as a function of the number of vertices is possible
(Corollary~\ref{c:doubleLog}).  The underlying
Theorem~\ref{t:tophomology} uses similar techniques as the proof of
\cite[Theorem~7]{CKV15}.  We use the following technical lemma,
the proof of which is a routine computation using the inequality
$(i-1)(i+1)<i^2$ several times.
\begin{lemma}\label{lem:technicalInequality}
For any integer $k\ge 3 $ we have
\[
\prod_{i=0}^{k-3} (k-i)^{2^i} < 12^{2^{k-3}}.
\]
\end{lemma}

\begin{theorem}\label{t:tophomology}
Let $\Delta$ be a simplicial complex of dimension $d$ on $n$ vertices
that is $(p+3)$-large for some $p\ge 2$, and has nontrivial top
homology.  If $f_i(\D)$ is the number of $i$-dimensional faces of $\D$, then
\begin{equation*}
f_d(\Delta) > \left(\frac{p^2+6p+9}{12}\right)^{2^{d-2}} \qquad\text{and}\qquad
f_0(\Delta) > \left(\frac{p^2+6p+9}{12}\right)^{2^{d-3}}.
\end{equation*}
\end{theorem}
\begin{proof}
For every $d$-dimensional simplicial complex $\Delta$ with nontrivial
top homology we define
\begin{eqnarray*}
v_{d}(\Delta) &=&  \min\{\textup{number of vertices in a top-dimensional cycle in }\Delta\},\\
s_{d}(\Delta) &=&  \min\{\textup{number of facets in a top-dimensional cycle in }\Delta\}.
\end{eqnarray*}
Minimizing over all $d$-dimensional $(p+3)$-large complexes with
nontrivial top homology, let
\begin{eqnarray*}
  v_{d} &=&  \min\{v_d(\Delta)~:~\Delta \textup{~$(p+3)$-large, with nontrivial top homology}\},\\
  s_{d} &=&  \min\{s_d(\Delta)~:~\Delta \textup{~$(p+3)$-large, with nontrivial top homology}\}.
\end{eqnarray*}
This implies in particular that $v_1 = s_1 = p+3$.

Fix a complex $\Delta$ satisfying the hypotheses of the theorem.  Let
$C\in\Delta$ be a top-dimensional cycle with $s_d(\Delta)$ facets.
For every vertex $v \in C$, the link $\lk_\Delta v$ is a
$(d-1)$-dimensional simplicial complex with nontrivial top homology by
item~\ref{it:cycleLemma1} in Lemma~\ref{l:cycles}. Furthermore,
$\lk_\Delta v$ is $(p+3)$-large.
Counting codimension one faces in $\Delta|_{V(C)}$ with multiplicity,
we get:
\[
s_d(\Delta) \ge \frac{1}{d+1}\sum_{v\in
V(C)}s_{d-1}(\lk_{\Delta|_{V(C)}} v).
\]
Fix $v\in V(C)$.  Every facet $F$ of the link of $v$ in
$\Delta|_{V(C)}$ is contained in at least two facets of $C$ only one
of which can contain~$v$. Thus, a map associating to
$F\in \lk_{\Delta|_{V(C)}}v$ a vertex $w\neq v$, with
$F\cup \{w\} \in \Delta$ is well defined:

\[
\Phi_v: \mathcal{F}(\lk_{\Delta|_{V(C)}}v) \longrightarrow
V(\Delta)\setminus V(\sta_{\Delta|_{V(C)}}v).
\]
We claim that $\Phi_v$ is injective.
To see this, let $F_1, F_2 \in \mathcal{F}(\lk_{\Delta|_{V(C)}}v)$ be
distinct faces such that $F_1\cup \{w\}$ and $F_2\cup \{w\}$ are faces
of~$\Delta$.  Since $\Delta$ is flag, there exist $v_1\in F_1$ and
$v_2\in F_2$ such that $v, v_1, w, v_2, v$ is a 4-cycle and
$\{v_1,v_2\}\notin \Delta$.
Since $\lk_{\Delta|_{V(C)}} v$ is flag, also $\{v,w\} \notin \Delta$.
Because of this contradiction, $\Phi_v$ is injective.
The injectivity yields $ v_d(\D)\geq s_{d-1}+ v_{d-1}+1$ and then,
putting together the above inequalities,
\[
s_{d}>\frac{ s_{d-1}^2}{d+1}, \qquad v_{d}> s_{d-1}.
\]
Now, since $s_{1}=p+3$,
\[
s_d> \frac{(p+3)^{2^{d-1}}}{\prod_{i=0}^{d-2}(d+1-i)^{2^{i}}}.
\]
Finally, by Lemma~\ref{lem:technicalInequality},
\[
f_d(\D)\geq s_d> \frac{(p+3)^{2^{d-1}}}{12^{2^{d-2}}}=\left(\frac{p^2+6p+9}{12}\right)^{2^{d-2}}.\qedhere
\]
\end{proof}

\begin{corollary}\label{c:doubleLog}
Let $I\subseteq R$ be a square-free monomial ideal such that $R/I$ is
a Cohen-Macaulay ring satisfying property $N_p$, for $p\ge 2$. Then
\[
\reg R/I \le \log_2 \log_{\frac{p^2+6p+9}{12}} n+3.
\]
\end{corollary}
\begin{proof}
Let $\D$ be a simplicial complex on $n$ vertices such that
$I=I_{\D}$. By Hochster's formula for local
cohomology~\cite[Theorem 13.13]{miller05:_combin_commut_algeb},
\[
\reg \kk[\D]=\max\{i:\widetilde{H}^{i-1}(\lk_{\D}\sigma;\kk)\neq 0:\sigma\in\D\}.
\]
Let $\sigma\in\D$ attain the maximum. Because $\D$ is Cohen-Macaulay,
$\lk_{\D}\sigma$ has nontrivial top homology. Therefore
$\reg \kk[\D]-1=\dim \lk_{\D}\sigma=:d$. Since $\D$ is $(p+3)$-large,
so is $\lk_{\D}\sigma$.  Hence by Theorem \ref{t:tophomology}
\[
n\geq f_0(\lk_{\D}\sigma) > \left(\frac{p^2+6p+9}{12}\right)^{2^{d-3}}
\]
and the conclusion follows.
\end{proof}

Corollary~\ref{c:doubleLog} motivates to ask
Question~\ref{generalquestion} again with a Cohen--Macaulay
restriction.  In this case the answer is not known.
\begin{question}\label{q:boundWithCM}
Fix an integer $p\geq 2$.  Is there a global bound $r(p)$ (independent of
$n$) such that $\reg R/I \le r(p)$ for all monomial ideals
$I\subset R$ for which $R/I$ satisfies~$N_p$ and is Cohen--Macaulay?
\end{question}

\section{Virtual cohomological dimension meets regularity}
\label{s:vcdmeetsreg}
The main theorem of this section establishes a new connection between
Coxeter groups and commutative algebra.  Its proof is by a cohomology
computation using two spectral sequences associated to a double
complex.  A reference and our source of notation is
\cite[Chapter~III]{GM03}.

Fix a ring $A$.  For any finite double complex
$L=(L^{p,q})_{(p,q)\in\N^2}$ of $A$-modules, there are two spectral
sequences both converging to the cohomology of the diagonal complex
$SL$ of $L$, whose entries are $SL^n = \oplus_{p+q=n}L^{p,q}$.
We denote these spectral sequences by
$(^I\!E_r^{p,q})$ and $(^{II}\!E_r^{p,q})$.  Both converge to
$^I\!E^k= \ ^{II}\!E^k=H^k(SL)$.  By \cite[III.7,
Proposition~10]{GM03}, $^I\!E_2^{p,q}$ is isomorphic to
$H_I^p(H_{II}^{\bullet,q}(L^{\bullet,\bullet}))$ (vertical cohomology
of horizontal cohomology), while $^{II}\!E_2^{p,q}$ is isomorphic to
$H_{II}^p(H_{I}^{q,\bullet}(L^{\bullet,\bullet}))$ (horizontal
cohomology of vertical cohomology).

For alignment with existing notation it is convenient to let $\D$ be a
simplicial complex with $n+1$ vertices~$V=\{0,\dots,n\}$.  For any
$s<n$ and any $i\in \{0,\ldots ,s\}$, denote by
$\D^i=\D|_{V\setminus \{i\}}$.  For any sequence of integers
$0\leq a_0<\ldots <a_p\leq s$, let
\[
\D^{a_0,\ldots,a_p}=\bigcap_{k=0}^p\D^{a_k}.
\]
Then $\D^{a_0,\ldots,a_p}$ equals the induced subcomplex
$\D|_{V\setminus \{a_0,\dots,a_p\}}$.  The first notation, however, is
more natural in following.  For example, if $\{0,\dots,s\}$ is not a
face of~$\D$, then $\{\D^i\}_{i=0,\dots,s}$ forms a closed cover of
$\D$, that is $\cup_{i=0}^s \D^i = \D$.  Denote by $\bfC^{\bullet}(\D,A)$
the cochain complex of a simplicial complex~$\D$ with coefficients in
the ring~$A$.  Consider the double complex of $A$-modules
$\C(A)=(\C^{p,q}(A))_{(p,q)\in\N^2}$ with
\begin{equation}\label{eq:double}
\C^{p, q}(A)=\bigoplus_{a_0<\ldots <a_p}\bfC^{q}(\D^{a_0,\ldots ,a_p};A),\qquad 0\le p \le s,\quad  0\le q \le \dim\D.
\end{equation}
where the direct sum runs over all sequences of $p+1$ integers
$0\leq a_0<\ldots <a_p\leq s$.  Throughout we use the standard
convention that all modules with indices outside of defined bounds are
zero.  The vertical maps $\C^{p,q}(A)\xrightarrow{}\C^{p,q+1}(A)$ are
just the maps defined for each direct summand in the cochain complex
$\bfC(\D^{a_0,\dots,a_p},A)$.  The rows
\begin{equation}\label{eq:horizontal}
0\rightarrow\C^{0, \bullet}(A)\xrightarrow{d^1} \C^{1,
\bullet}(A)\xrightarrow{d^2} \cdots \xrightarrow{d^{s-1}}
\C^{s-1, \bullet}(A)\xrightarrow{d^s} \C^{s, \bullet}(A)\rightarrow 0,
\end{equation}
are defined by mapping an element
$\alpha=(\alpha_{a_0,\ldots,a_p})_{a_0<\ldots<a_p} \in \C^{p,q}(A)$,
to $d^{p+1}(\alpha)\in \C^{p+1,q}(A)$, whose $(b_0,\ldots ,b_{p+1})$th
component is
\[
\sum_{k=0}^{p+1}(-1)^{k}
\left.
\left(\alpha_{b_0,\ldots ,\widehat{b_k},\ldots ,b_{p+1}}\right)
\right|_{\bfC^q(\D^{b_0,\ldots ,b_{p+1}})}.
\]
A routine computation confirms that this defines a double complex.
The vertical cohomology is by definition the direct sum of the
cohomologies of the corresponding $\D^{a_0,\dots,a_p}$.  The
horizontal cohomology is nontrivial only in cohomological degree~0
according to the following lemma, whose proof is standard.

\begin{lemma}\label{lem:horizontalCohomology}
For each $q \in \{0,\dots,\dim\D\}$ we have
\[H^p(\C^{\bullet ,q}(A))= 
\begin{cases}
\bfC^{q}(\cup_i\D^i ;A)& \text{if } p = 0,\\
0 & \text{if } p>0.

\end{cases}
\]
\end{lemma} 
\begin{theorem}\label{t:cd=reg}
Let $(W,S)$ be a Coxeter group and $\NN$ its nerve. Then
\[\vcd W=\max_{\chara \kk}\{\reg \kk[\NN] \}.\]
\end{theorem}

\begin{proof}
In Section \ref{sec:preliminaries}, equations~\eqref{eq:vcd}
and~\eqref{eq:reg1} present interpretations of both invariants in
terms of the reduced simplicial cohomology of $\NN$, namely
\begin{align*}
\vcd W & =\max\{i:\widetilde{H}^{i-1}(\NN|_{S\setminus \sigma};\Z)\neq 0 {\mbox{ for some } \sigma\in \NN}\},\\
\reg \kk[\NN] &=\max\{i:\widetilde{H}^{i-1}(\NN|_{U};\kk)\neq 0 \mbox{ for some }U\subseteq S \}.
\end{align*}
Therefore the result is a consequence of the following claim.

\noindent
\textit{Claim}. Let $\D$ be a simplicial complex on
$V=\{0,\ldots ,n\}$ and $A$ be a ring. Then
\begin{equation}\label{e:cd=reg}
\begin{split}
& \max\{i:H^i(\D|_{V\setminus \sigma};A)\neq 0 \mbox{ for some
}\sigma\in\D\}= \qquad \qquad \qquad \\
& \qquad \qquad \qquad = \max\{i:H^i(\D|_{V'};A)\neq 0 \mbox{ for some
}V'\subseteq V\}.
\end{split}
\end{equation}
Clearly the left-hand side is less than or equal to the the right-hand
side. To see that equality holds, let $r$ be the maximum on the right
and choose $V'\subseteq V$ such that $H^r(\D|_{V'};A)\neq 0$. If
$V\setminus V' \in \D$ we have nothing to prove, so assume that
$V\setminus V' \notin \D$ (in particular $|V\setminus V'|\ge 2$). We can
(and will) also assume that $H^i(\D|_{U};A)=0$ for all $i\geq r$ and
$V'\subsetneq U\subseteq V$.

After a potential renumbering we can assume that
$V\setminus V'=\{0,\ldots ,s\}$.  For any $i\in \{0,\ldots ,s\}$, let
$\D^i=\D|_{V\setminus \{i\}}$ and consider the double complex defined
in~\eqref{eq:double}.
By Lemma~\ref{lem:horizontalCohomology} $(^I\!E_r^{p,q})$ stabilizes
at the second page and
\[
H^p(\cup_{i=0}^s\Delta^i;A)= \ ^I\!E^{p,0}_2= \ ^I\!E^{p,0}_{\infty}= \ ^I\!E^p.
\]
Since $\{0,\ldots ,s\}$ is not a face of~$\D$, we have
$\D=\cup_{i=0}^s\Delta^i$.  Now consider the spectral
sequence~$(^{II}\!E_r^{p,q})$.
From the maximality assumption on $V'$ it follows that
\[
H^r(\D|_{V'};A)=H^r(\D^{0,\ldots ,s};A)= \ ^{II}\!E^{r,s}_2.
\]
In particular, if $r'>r$ or $s'>s$, then $^{II}\!E^{r',s'}_2=0$, since
it is a subquotient of
\[\bigoplus_{0\leq a_0<\ldots <a_{s'}\leq s}H^{r'}(\D^{a_0,\ldots ,a_{s'}};A)=0.\]
We have
$^{II}\!E^{r,s}_2= \ ^{II}\!E^{r,s}_{\infty}= \ ^{II}\!E^{r+s}$ from
which we conclude that
\[
H^{r+s}(\D;A)= \ ^I\!E^{r+s}= \ ^{II}\!E^{r+s}= H^r(\D|_{V'};A)\neq 0.
\]
Since $s>0$ (because $|V\setminus V'|\ge 2$), we obtain a contradiction to the maximality of $r$ and~$V'$.
\end{proof}

\section{Arbitrary large regularity with property $N_{p}$}
\label{s:InductiveConstruction}

We now prove Theorem~\ref{t:generalThmFirstStatement}.  To this end,
for each $k$-large simplicial complex $\Delta$ we construct a
$k$-large simplicial complex $\ritter(\Delta,k)$ such that, in
characteristic zero,
$\reg\kk[\ritter(\Delta,k)] = \reg\kk[\Delta] + 1$
(Lemma~\ref{l:regularityUp}).  This uses a construction based on a
detour through geometric group theory and is inspired by the work of
Osajda~\cite[Section~4]{osajda2010construction}.

We need to make a few definitions.  
The first turns a cell complex into a simplicial complex.
\begin{definition}\label{def:thickening}
The \emph{thickening} of a convex cell complex $\P$ is the simplicial
complex $\Th(\P)$, with the same vertex set as $\P$, obtained by turning
all cells into simplices. In particular, $\{v_1,\dots,v_s\}$ is a face
of $\Th(\P)$, if there is a face of $\P$ that contains $\{v_1,\dots,v_s\}$.
\end{definition}

\begin{example}
The thickening of the $d$-dimensional cube is the $(2^d-1)$-simplex.
\end{example}
The thickening induces a distance between the vertices of a convex
cell complex that counts the minimal number of maximal cells one needs
to pass to get from one vertex to another. Namely, for two vertices
$v,w\in \P$, the distance $d(v,w)$ is the length of a shortest path
connecting $v$ and $w$ in the 1-skeleton of the thickening $\Th(\P)$.

A step in our construction is taking a finite quotient of an infinite
cubical complex. We clarify here how this is intended.  Let $G$ be a
group acting on the vertex set $V(\P)$ of a convex cell complex $\P$
such that for every face $F=\{v_1,\dots,v_k\} \in \P$ and every
$g\in G$ we have
\[g\cdot F = \{g\cdot v_1,\dots,g\cdot v_k\} \in \P.\]
This induces an action of $G$ on $\P$. The \emph{displacement} of the
action of $G$ on $\P$ is the minimum distance between the elements in
the orbit of a vertex. We can take the quotient $\P/G$, which is in
general only a set.
\begin{remark}\label{r:quotientComplex}
If the displacement of the action is at least $2$, then $\P/G$ is a
poset with the inclusion given by $\widehat{F'}\subseteq \widehat{F}$
if there exists $g\in G$ such that $g\cdot F' \subseteq F$. If the
displacement of the action is at least $3$, then $\P/G$ is a convex
cell complex.
\end{remark}
An example of such a group action is that of the subgroup of some
Coxeter group on the vertices of the Davis complex. In this case the
displacement of the action coincides with the displacement of the
subgroup as defined below.
\begin{definition}
Let $W$ be a Coxeter group. The \emph{displacement} of an element
$w\in W$ is the distance $d(e,w)$ of $w$ to the identity in the
(1-skeleton of the) thickening~$\Th(\Sigma)$.  The \emph{displacement
of a subgroup $H \subset W$} is the minimal displacement among its
nontrivial elements.
\end{definition}

Let $\Delta$ be a $k$-large simplicial complex for an integer
$k\ge 4$.  We introduce an iterative construction which produces a new
$k$-large simplicial complex $\ritter(\Delta,k)$.  It works as
follows.
\begin{enumerate}
\item Let $W$ be the right-angled Coxeter group with nerve~$\Delta$.
\item Let $\Sigma$ be the Davis complex of $W$.
\item Let $Y = \Th(\Sigma)$ be the thickening of $\Sigma$.
\item \label{it:pick} Pick a torsion-free finite index subgroup
$H\subset W$ with displacement at least~$k$.
\item \label{it:quotient} Let $\ritter(\Delta,k)$ be the
quotient~$Y/H$.
\end{enumerate}

Since $\Delta$ is flag, there is a right-angled Coxeter group
$\Ww(\D)$ as described in
Remark~\ref{r:flagRight-angledCorrespondence}.
The group $H$ in~\ref{it:pick} exists because $W$ is virtually torsion
free~\cite[Corollary~D.1.4]{davis2008geometry} and residually
finite~\cite[Section~14.1]{davis2008geometry}.  In
Section~\ref{sec:bound-numb-vert}, we take a constructive approach and
find a concrete $H$ using representations of $W$ in $\GL_n(\Z)$.  The
resulting complex $\ritter(\Delta,k)$ evidently depends on the choice
of $H$ in step~\ref{it:pick}.  However, the desired properties of
$\ritter(\Delta,k)$, such as Lemma~\ref{l:regularityUp}, do not depend
on this choice.
\begin{lemma}\label{l:faceVSthick} In the above situation,
$\face{\D} = \lk_{\ritter(\Delta,k)} v$ for any vertex~$v$.
\end{lemma}
\begin{proof}
After unraveling definitions, it is visible that if $\Sigma$ is a
cubical complex and $v\in \Sigma$ is a vertex, then
$\face{\lk_\Sigma v} = \lk_{\Th(\Sigma)} v$.  If $\Sigma$ is the Davis
complex of a right-angled Coxeter group with nerve $\Delta$, then by
Remark~\ref{r:LinkNerve}, $\lk_{\Sigma}v = \Delta$ for any vertex
$v\in \Sigma$.
\end{proof}

\begin{lemma}\label{l:regularityUp} If $\kk$ is a field of
characteristic zero and $k\ge 4$, then
$\reg \kk[\ritter(\Delta,k)] = \reg\kk[\D]+1$.
\end{lemma}
\begin{proof}
By the previous lemma $\face{\D} = \lk_{\ritter(\Delta,k)} v$ for
any vertex $v$, and by Proposition~\ref{p:reglink} there exists a
vertex $v$ such that
$\reg \kk[\lk_{\ritter(\Delta,k)}v] \ge \reg
\kk[\ritter(\Delta,k)]-1$.  Since
$\reg \kk[\face{\D}]= \reg\kk[\D]$ by
Proposition~\ref{p:regFaceEqual}, it follows that
$\reg \kk[\ritter(\Delta,k)] \leq \reg\kk[\D] + 1$.

To show $\reg \kk[\ritter(\Delta,k)] \ge \reg\kk[\D]+1$, let
$X=\Sigma/H$.  Then $\ritter(\Delta,k)$ is the thickening of~$X$. By
Hochster's formula for graded Betti numbers and~\eqref{e:cd=reg}, we
have that
\[
\reg \kk[\D]=\max\{i:\widetilde{H}^{i-1}(\D\setminus \sigma;\kk)\neq
0 \text{ for some }\sigma\in\D\}.
\]
Let $r=\reg \kk[\D]$, and fix $\sigma\in\D$ for which
$\widetilde{H}^{r-1}(\D\setminus \sigma;\kk)\neq 0$.  From now on the
argument goes on the same lines of the proof leading to \cite[Lemma
4.5]{osajda2013combinatorial}.  With the same notation used there,
$\D\setminus \sigma$ deformation retracts onto
$K^{S\setminus \sigma}$, where $K$ is the subcomplex of $\Sigma$
induced by the spherical words (including the identity) and, for any
subset of generators $T\subseteq S$, $K^T$ is the subcomplex induced
by the spherical words containing some element of $T$.  So we have
\[
\widetilde{H}^{r}(K,K^{S\setminus \sigma};\kk)\neq 0.
\]
Osajda produces a map of $\kk$-vector spaces from the cocycles
$Z^r( K,K^{S};\kk)$ to the cocycles $Z^r(X;\kk)$.  This uses the
assumption $\chara(\kk) = 0$.  One can check that the same rule
defines a map of $\kk$-vector spaces
$Z^r( K,K^{S\setminus \sigma};\kk) \to Z^r(X\setminus A;\kk)$, where
$A=\{\widehat{w\sigma}:w\in W\}$ and $\widehat{w\sigma}$ is the class
in $X$ of $w\sigma\in\Sigma$.  By the same argument used in
\cite[Lemma 4.5]{osajda2013combinatorial}, the above map induces an
injection
\[
\widetilde{H}^r( K,K^{S\setminus \sigma};\kk)\hookrightarrow
\widetilde{H}^r(X\setminus A;\kk),
\]
in particular $\widetilde{H}^r(X\setminus A;\kk)$ is not zero.  By
\cite[Lemma~70.1]{munkres1984elements},
$\widetilde{H}^k(X\setminus A;\kk)\cong \widetilde{H}^r(X_B;\kk)$,
where $B$ are the vertices of $X$ which are not in $\widehat{w\sigma}$ for
any $w\in W$. Finally, the thickening of $X_B$ is exactly
$\ritter(\D,k)_B$, so
\[\widetilde{H}^r(\ritter(\D,k)_B;\kk) \neq 0.\]
By Hochster's formula for graded Betti numbers
$\reg \kk[\ritter(\Delta,k)]\geq r+1$.
\end{proof}
 
\begin{remark}
In the definition of cohomological dimension, $\ZZ$ could be replaced
by a field $\kk$ of characteristic zero.  The resulting notion of {\it
virtual rational cohomological dimension} $\vcd_\Q W$ of a virtually
torsion free group $W$ does not depend on the choice of the field.
This notion however, differs from virtual cohomological dimension.
Lemma~\ref{l:regularityUp}, together with Hochster's formula for
graded Betti numbers and~\eqref{e:cd=reg}, implies that
\[\vcd_\Q \Ww(\ritter(\Delta,k)) = \vcd_\Q \Ww(\D)+1.\]
This conclusion for $\vcd$ does not follow from
Lemma~\ref{l:regularityUp} because of the assumptions on~$\kk$.
\end{remark}

\begin{lemma}\label{l:locallyklarge}
If a cubical complex is locally $k$-large, then its thickening is
locally $k$-large.
\end{lemma}
\begin{proof}
Let $\Sigma$ be a locally $k$-large cubical complex.  As in the proof of
Lemma~\ref{l:faceVSthick}, each vertex link $\lk_{\Th(\Sigma)} v$ is equal
to $\face{\lk_\Sigma v}$.  By a result of Haglund, a simplicial complex is
$k$-large if and only if its face complex is
$k$-large~\cite[Proposition~B.1]{januszkiewicz10:_non}.
\end{proof}
A proof of the Lemma~\ref{l:locallyklarge} also appears
in~\cite[Lemma~6.7]{osajda2013combinatorial}.

\begin{lemma}\label{l:thickeningklarge}
Let $\Sigma$ be the Davis complex of $\Ww(\Delta)$, where $\Delta$ is
$k$-large for $k\ge 4$.  Then $\Th(\Sigma)$ is $k$-large.
\end{lemma}
\begin{proof}
The Davis complex $\Sigma$ is a deformation retract of its thickening
$\Th(\Sigma)$ and in particular has the same homotopy type.  Therefore
$\Th(\Sigma)$ is simply connected.  By Lemma~\ref{l:locallyklarge}
$\Th(\Sigma)$ is locally $k$-large.  According to
\cite[Corollary~1.5]{JS-non-positive-curvature}, a simplicial complex
is $k$-large if and only if all links are $k$-large and the systole
(the length of the shortest non-contractible loop in the complex) is
at least~$k$.  Since there are no non-contractible loops, the proof is
complete.
\end{proof}

When forming the quotient of the thickening of the Davis complex
modulo the finite-index torsion-free subgroup $H\subset W$ in
step~\ref{it:quotient} of the construction, cycles are created.  The
quotient by a group of displacement $k$ creates cycles of length~$k$.
By Remark \ref{r:quotientComplex}, $k \ge 4$ implies the quotient is
simplicial complex.

\begin{lemma}\label{l:QuotDisplace}
Let $\Sigma$ be the Davis complex of $\Ww(\Delta)$, where $\Delta$ is
$k$-large for $k\ge 4$.  If $H\subset \Ww(\Delta)$ is a torsion-free
subgroup of displacement at least~$k$, then $\Th(\Sigma)/H$ is
$k$-large.
\end{lemma}
\begin{proof}
If $C \in \Th(\Sigma)/H$ is a cycle of length $l < k$, then it
consists of disjoint orbits and thus there is a cycle of length $l$ in
$\Th(\Sigma)$.  This is impossible since by
Lemma~\ref{l:thickeningklarge} $\Th(\Sigma)$ is $k$-large.
\end{proof}

We are now ready to prove the two main results of this section.
\begin{theorem}\label{t:generalThmFirstStatement}
Let $I=I_{\D}\subseteq R = \kk[x_1,\dots,x_n]$ be a square-free
quadratic monomial ideal. If the characteristic of $\kk$ is zero, then
there exists a positive integer $N$ and a square-free monomial ideal
\mbox{$I'=I_{\D'}\subseteq R' = \kk[y_1,\dots,y_N]$} such that:
\begin{enumerate}[label=\textup{\roman*)},ref=\roman*]
\item $\reg R'/I'= \reg R/I+1$;
\item $\ind R'/I'=\ind R/I$;
\item For each vertex $v$ of $\D'$, $\lk_{\D'}v=\face{\D}$.
\end{enumerate}
\end{theorem}

\begin{proof}
Let $p = \ind R/I$. Since $I=I_\D$ is quadratic, $p\ge 1$.  Let
$\D' = \ritter(\Delta,p+3)$.  The first item is
Lemma~\ref{l:regularityUp}.  The third item is
Lemma~\ref{l:faceVSthick}.  For the second item,
Lemma~\ref{l:thickeningklarge} implies that $\ind R'/I' \ge p$.  If
$\ind R'/I' > p$, then $\ind \kk[\face{\D}] > p$ since if there are no
induced $(p+3)$-cycles, then no link in $\D$ has an induced
$(p+3)$-cycle.  Furthermore, by
\cite[Proposition~B.1]{januszkiewicz10:_non},
$\ind \kk[\face{\D}] = \ind \kk[\Delta] = \ind R/I$.
\end{proof}

\begin{corollary}\label{c:regularityArbitrary}
For any positive integers $p$ and $r$, there exists a square-free
monomial ideal $I\subseteq R = \kk[x_1,\ldots ,x_{N(p,r)}]$, such that
$R/I$ satisfies $N_p$ and $\reg R/I = r$.
\end{corollary}

\begin{proof}
Let $\Delta_2$ be the $(p+3)$-cycle, and inductively
$\Delta_r = \ritter (\Delta_{r-1}, p+3)$.  Then $\Delta_r$ satisfies
the conditions of the corollary if $\chara \kk = 0$.  To see that the
construction is independent of the field, assume that for some $\kk$,
$\reg \kk[\D_r] > \reg\Q[\D_r]$.  By Lemma~\ref{l:faceVSthick},
Proposition~\ref{p:reglink}, and Proposition~\ref{p:regFaceEqual},
$\reg \kk[\D_{r-1}] > \reg\Q[\D_{r-1}]$ and inductively
$\reg\kk[\D_2] > \reg\Q[\Delta_2]$ which is not the case.
\end{proof}

\begin{remark}\label{r:JS19.2}
In \cite[Corollary 19.2]{JS-non-positive-curvature}, Januszkiewicz and
\swi proved, for any $k\geq 6$ and $d\in\N$, the existence of a
$k$-large orientable $d$-dimensional pseudomanifold.  Together with
\cite[Theorem~1]{januszkiewicz2003hyperbolic} this could be used to
give a shorter proof of Corollary~\ref{c:regularityArbitrary}.  We
feel that such a proof would have been less insightful for commutative
algebra.
\end{remark}

\begin{remark}\label{r:NP}
The results in this section can also be used to strengthen a result of
Nevo and Peeva who studied a question of Francisco, H\`a and Van Tuyl.
The latter noticed (unpublished) that if $I\subset R$ is a quadratic
square-free monomial ideal such that $I^s$ has a linear resolution for
all $s\geq 2$, then $R/I$ satisfies $N_2$, and wondered if the
converse was true. In \cite[Counterexample~1.10]{NP13} Nevo and Peeva
gave a square-free monomial ideal $I\subset R$ such that $R/I$ has
property $N_2$ but $I^2$ does not have a linear resolution.  Using our
results, this can be extended to $N_p$ and any power as follows.
\begin{corollary}
For any integers $p,t\ge 2$ there exists a square-free monomial ideal
$I\subset R$ such that $R/I$ has property $N_{p}$ and $I^s$ does not
have a linear resolution for all $1\leq s\leq t$.\end{corollary}
\begin{proof}
Set $r=2t$ and choose $I\subset R$ as in
Corollary~\ref{c:regularityArbitrary}.  Then $\reg(R/I) = 2t$, and
$\reg(R/I^s)\geq 2t$ for all $s\geq 1$ by Proposition~\ref{p:dubois}.
\end{proof}
Question~1.11 in \cite{NP13} asks whether $I_{\D}^s$ has a linear
resolution for $s\gg 0$ whenever $\kk[\D]$ satisfies $N_2$.  It remains
open and the construction yielding
Theorem~\ref{t:generalThmFirstStatement} provides examples worth
testing.  For an experimental investigation with computer algebra, the
number of variables involved would need a vast improvement, though.
\end{remark}

\section{Counting the number of vertices of $\ritter(\Delta,k)$}
\label{sec:bound-numb-vert}
For complexity theory in commutative algebra a bound on the number of
variables $N(r,p)$ in Corollary~\ref{c:regularityArbitrary} is
necessary.  We now derive such a bound by controlling the choice of
the torsion-free subgroup $H$ in step~\ref{it:pick} of the
construction of $\ritter(\Delta,k)$.

Each Coxeter group $W$ can be embedded in $\GL_n(\mathbb{R})$ by means
of its canonical representation
$\rho : W \to
\GL_n(\mathbb{R})$~\cite[Corollary~6.12.4]{davis2008geometry}.  This
representation starts from the cosine matrix $C = (c_{ij})_{ij}$ of a
Coxeter system whose entries are $c_{ij} = -\cos(\pi/m_{ij})$.  A
generator $s_{i}$ is represented by the linear map
$\rho(s_i) : x \mapsto x - 2 \sum_jc_{ij}x_j e_i$.  As the order of
every product of generators is $2$ or~$\infty$, right-angled Coxeter
groups embed also in $\GL_n(\mathbb{Z})$.
More specifically, since the cosine matrix has entries only
$-1, 0, 1$, the canonical representation matrices use only
$0, \pm 1, 2$.  An easy computation using the definition of the linear
map for one generator and $\cos(\pi/2) = 0$ shows that whenever
$w = s_{i_1}\cdots s_{i_l}$ is a spherical word, then it is
represented by the linear map
\begin{equation}\label{e:sphericalMatrix}
\rho(w) : x \mapsto x - 2\sum_jc_{i_1j}x_j e_{i_1} - \cdots - 2\sum_jc_{i_lj}x_je_{i_l}.
\end{equation}
We thus showed a simple fact about the entries of~$\rho(w)$.

\begin{lemma}\label{l:sphericalWords}
Let $\Ww(\D)$ be a right-angled Coxeter group with nerve $\D$ and
$d=\dim\D$.  For each spherical word $w \in\Ww(\D)$ of length $l$, the
matrix $\rho(w)$ uses only $0, \pm 1, 2$ for its entries and each of
its columns has at most $l$ entries equal to two.
\end{lemma}
We employ the projection $\GL_n (\Z) \to \GL_n(\Z/m\Z)$ to
find finite-index torsion-free subgroups $H$ as in step~\ref{it:pick}
of the construction in Section~\ref{s:InductiveConstruction}, so that
the size of $\ritter(\Delta, k)$ can be controlled.  To preserve
$k$-largeness, we need to choose $m$ so that no words of displacement
$< k$ reduce to the identity modulo~$m$.  This requires information
about the orders of elements of~$\GL_n(\Z/m\Z)$.

Fix $k>4$ and a $k$-large simplicial complex $\Delta$ of dimension~$d$
 with $n$ vertices.  For any $m\ge 2$ consider the canonical
homomorphism
\[
\pi_m: \GL_{n}(\Z)\to \GL_{n}(\Z/m\Z).
\]
Denote $\Gamma_m = \Ker(\pi_m)$ and let
$\Xi_m = \Gamma_m \cap \rho(\Ww(\Delta)) \subseteq \rho(\Ww(\Delta))$ be the
subgroup of $\rho(\Ww(\Delta))$ that lies in the kernel of $\pi_m$.
\begin{lemma}\label{l:subgrTorFree}
$\Xi_m$ is torsion free if $m > 2$.
\end{lemma}
\begin{proof}

It is well known that any torsion element in a right-angled Coxeter
group has order two and is in fact conjugate to a spherical word.  Let
$w\in \Xi_m$ be an involution and write $w = g^{-1}s g$ with some
spherical word~$s$ and $g\in \Ww(\Delta)$.  Then
$1 = \pi_m(w) = \pi_m(g)^{-1} \pi_m(s) \pi_m(g)$ implies
$\pi_m(s) = 1$ which for $m>2$ implies $s=1$ (by
Lemma~\ref{l:sphericalWords}) and finally~$w=1$.
\end{proof}
The subgroup to be used in step \ref{it:pick} is
$H(m) = \rho^{-1}(\Xi_m)$.  Let $w\in \Ww(\Delta)$.
As a function of the displacement and the dimension $d$ of~$\Delta$,
we determine an upper bound on
$a(w) = \max\{ |\rho(w)_{i,j}| : 1\le i,j \le n\}$, the maximum
absolute value of the entries of the corresponding matrix
$\rho(w) \in \GL_n(\Z)$.

\begin{lemma}\label{l:boundEntry2}
Let $w$ be a word of displacement less than $k$, then
$a(w) < (2d+3)^{k-1}$.
\end{lemma}
\begin{proof}
A word of displacement less than $k$ is a product of at most $k-1$
spherical words.  When $w$ is a spherical word, it has length at most
$d+1$, and thus each column of $\rho(w)$ has at most $d+1$ entries $2$
and one entry~$1$ by Lemma~\ref{l:sphericalWords}.  This yields the
recursion $a(ws) \le (2d+3)a(w)$.  Since $a(s) = 2$ for any spherical
word, the bound follows.
\end{proof}
Our aim is to pick an integer $m$ so that any word in $H(m)$ has
displacement at least~\mbox{$p+3$}.  Lemma~\ref{l:boundEntry2} shows
that $m= (2d+3)^{p+2}$ is sufficient.  Given $m$, the number of
vertices of $\ritter(\Delta,p)$ is bounded by the size of
$\GL_{n}(\Z/m\Z)$ which is of the order~$m^{n^2}$.  Iterating the
construction of $\ritter(\Delta,p)$, we achieve the desired bound for
the number of variables needed in
Corollary~\ref{c:regularityArbitrary}.  To write it, we use Knuth's up
arrow notation~\cite{knuth1976mathematics} which is convenient for
iterative constructions.  Fortunately we can limit ourselves to two up
arrows which represent power towers.  Specifically,
$a\uparrow\uparrow b$ means $a^{a^{a^{\dots}}}$ exactly $b$ times.
\begin{theorem}\label{t:boundOnVertices}
For all $p$, there exists a family of ideals indexed by $r$ realizing 
Corollary~\ref{c:regularityArbitrary} with
\[
N(p,r+1) < (2(2\uparrow\uparrow(r-1)) + 1)^{(p+2)N(p,r)^2}.
\]
Furthermore, if $c_p$ is the smallest integer such that $2 \uparrow\uparrow c_{p} > p+2$, then
\[
N(p,r+1) < 2\uparrow\uparrow (r(r+c_p)).
\]
\end{theorem}

\begin{proof}
Let $\Delta_2$ be the $(p+3)$-cycle which implies $N(p,2) = p+3$.  Let
$\Delta_{r+1} = \ritter (\Delta_{r}, p)$, where the subgroup in step
\ref{it:pick} is chosen as $H(m_{r+1})$ with
$m_{r+1} = (2d_r+1)^{p+2}$.  Here $d_r = \dim \Delta_r + 1$ and thus
$d_2 = 2$.  We have the recursion $d_{r+1} = 2^{d_r}$, which yields
$d_{r} = 2\uparrow\uparrow (r-1)$.  The number of vertices of
$\Delta_{r+1}$ is bounded by the order of $\GL_n(\Z/m_{r+1}\Z)$.
Estimating this order as $m_{r+1}^{N(r,p)^2}$ we obtain the recursive bound.

For  the second part  we use the fact that removing parenthesis from a
power tower does not make the expression smaller by generalizations of
$\left(2^2\right)^{(2^2)} < 2^{2^{2^2}}$. We thus get
\begin{align*}
  N(p,r+1) & < (2\uparrow\uparrow r)^{(p+2)N(p,r)^2} \\
           & < (2\uparrow\uparrow r)^{(2\uparrow \uparrow c_p)
	     N(p,r)^2} \\
           & < (2 \uparrow \uparrow (r + c_p))^{N(p,r)^2}.
\end{align*}
Now by a simple induction, the structure of the expression on the
right is continued exponentiation of $2$ for at most $r(r+c_p)$ times,
but with certain parenthesis inside the tower.  Removing the
parentheses we conclude.
\end{proof}
We hope that the bound in Theorem~\ref{t:boundOnVertices} can be
improved significantly.  To justify this hope we illustrate vast
improvements in a simple example.  Let $\Delta$ be the 5-cycle.
The right-angled Coxeter group with nerve $\Delta$ has the following
Coxeter and cosine matrices
\[
\begin{pmatrix*}[c]
1	& 2	  & \infty  & \infty  & 2	\\
2	& 1	  & 2	    & \infty  & \infty	\\
\infty	& 2	  & 1	    & 2	      & \infty	\\
\infty	& \infty  & 2	    & 1	      & 2	\\
2	& \infty  & \infty  & 2	      & 1	\\
\end{pmatrix*}, \qquad
C = 
\begin{pmatrix*}[r]
1   & 0	  & -1	& -1  & 0   \\
0   & 1	  & 0	& -1  & -1  \\
-1  & 0	  & 1	& 0   & -1  \\
-1  & -1  & 0	& 1   & 0   \\
0   & -1  & -1	& 0   & 1   \\
\end{pmatrix*}.
\]
The generators of the standard representation of this Coxeter group
are
\[
s_1 \mapsto
\begin{pmatrix*}[r]
-1  & 0	& 2 & 2	& 0   \\
0   & 1	& 0 & 0	& 0   \\
0   & 0	& 1 & 0	& 0   \\
0   & 0	& 0 & 1	& 0   \\
0   & 0	& 0 & 0	& 1   \\
\end{pmatrix*}
\qquad \dots \qquad
s_5 \mapsto \begin{pmatrix*}[r]
1   & 0	& 0 & 0	& 0   \\
0   & 1	& 0 & 0	& 0   \\
0   & 0	& 1 & 0	& 0   \\
0   & 0	& 0 & 1	& 0   \\
0   & 2	& 2 & 0	& -1  \\
\end{pmatrix*}.
\]
Table~\ref{tab:words} gives the maximum absolute value of entries of
words of length~$l$ in $\Ww(\Delta)$.  It shows that the prime number
$p=1811$ would certainly suffice to guarantee that no word of
displacement $\le 5$ (which all have length $\le 10$) is in the kernel
of the reduction modulo~$p$.
\begin{table}[htpb]
\centering
\begin{tabular}{|l||c|c|c|c|c|c|c|c|c|c|}
  \hline
  word length & 1 & 2 & 3 & 4	& 5   & 6   & 7   & 8	& 9   & 10    \\
  \hline
  entry size  & 2 & 4 & 8 & 18  & 39  & 84  & 180 & 388 & 836 & 1801  \\
  \hline
\end{tabular}
\caption{Entry sizes in words\label{tab:words}}
\end{table}
However, it can be checked algorithmically (we used the Coxeter group
functionality in \textsc{sage}~\cite{sage}) that no word of length at
most $10$ is in the kernel of the reduction modulo~$7$.  In the
reduction modulo~$5$, however, $(s_1s_3)^5$ maps to the identity.  We
also checked words of length $12$ for the Coxeter group corresponding
to the heptagon.  There $7$ is not large enough, as for example
$(s_1s_3s_1s_5)^3$ goes to the identity.

In the example of the 5-cycle, the bound derived in
Theorem~\ref{t:boundOnVertices} yields $N(2,3) < 5^{100}$, while using
$m=7$ yields~$N(2,3)<7^{25}$.  In contrast one can exhibit a 5-large 
triangulation of a 2-sphere with 12 vertices.
Nevertheless, a good understanding of representations of Coxeter
groups in finite characteristic should yield better estimates than
Theorem~\ref{t:boundOnVertices}.

The integer $m_r$ used in the recursive construction of $\Delta_r$ in
Theorem~\ref{t:boundOnVertices} currently depends on the dimension
which grows very quickly.  It is conceivable that for each $p$ there
is a uniform bound, independent of $r$.
\begin{question}
Is there a bound for the integer $m_r$ that depends only on $p$ and
not on~$r$?
\end{question}
\bibliographystyle{amsalpha}
\bibliography{flagnosquare}

\newcommand{\etalchar}[1]{$^{#1}$}
\providecommand{\bysame}{\leavevmode\hbox to3em{\hrulefill}\thinspace}
\providecommand{\MR}{\relax\ifhmode\unskip\space\fi MR }
\providecommand{\MRhref}[2]{%
  \href{http://www.ams.org/mathscinet-getitem?mr=#1}{#2}
}
\providecommand{\href}[2]{#2}
\begin{thebibliography}{CCM{\etalchar{+}}19}

\bibitem[ACI13]{avramov2013subadditivity}
Luchezar~L Avramov, Aldo Conca, and Srikanth~B Iyengar, \emph{Subadditivity of
  syzygies of {K}oszul algebras}, Mathematische Annalen \textbf{361} (2013),
  no.~1-2, 511--534.

\bibitem[AH16]{ananyanStillmanC}
T.~Ananyan and M.~Hochster, \emph{Small subalgebras of polynomial rings and
  {S}tillman's conjecture}, preprint, arXiv:1610.09268 (2016).

\bibitem[BS88]{bayer1988complexity}
David Bayer and Michael Stillman, \emph{On the complexity of computing
  syzygies}, Journal of Symbolic Computation \textbf{6} (1988), no.~2,
  135--147.

\bibitem[BS13]{BS13}
M.~P. Brodmann and R.~Y. Sharp, \emph{Local cohomology}, second ed., Cambridge
  Studies in Advanced Mathematics, vol. 136, Cambridge University Press,
  Cambridge, 2013, An algebraic introduction with geometric applications.

\bibitem[CCM{\etalchar{+}}19]{CMPV17}
Giulio Caviglia, Marc Chardin, Jason McCullough, Irena Peeva, and Matteo
  Varbaro, \emph{Regularity of prime ideals}, Mathematische Zeitschrift
  \textbf{291} (2019), 421–435.

\bibitem[CKV16]{CKV15}
Alexandru Constantinescu, Thomas Kahle, and Matteo Varbaro, \emph{Linear
  syzygies, flag complexes, and regularity}, Collectanea Mathematica
  \textbf{67} (2016), no.~3, 357--362.

\bibitem[Cox35]{Coxeter1935}
H.~S.~M. Coxeter, \emph{The complete enumeration of finite groups of the form
  {$R_i^2 = (R_iR_j)^{k_{ij}} = 1$}}, Journal of the London Mathematical
  Society \textbf{s1-10} (1935), no.~1, 21--25.

\bibitem[CS05]{caviglia2005characteristic}
Giulio Caviglia and Enrico Sbarra, \emph{Characteristic-free bounds for the
  {C}astelnuovo--{M}umford regularity}, Compositio Mathematica \textbf{141}
  (2005), no.~6, 1365--1373.

\bibitem[Dav08]{davis2008geometry}
Michael Davis, \emph{The geometry and topology of {C}oxeter groups}, vol.~32,
  Princeton University Press, 2008.

\bibitem[DHS13]{dao2013bounds}
Hailong Dao, Craig Huneke, and Jay Schweig, \emph{Bounds on the regularity and
  projective dimension of ideals associated to graphs}, Journal of Algebraic
  Combinatorics \textbf{38} (2013), no.~1, 37--55.

\bibitem[DJ91]{davis1991convex}
Michael Davis and Tadeusz Januszkiewicz, \emph{Convex polytopes, {C}oxeter
  orbifolds and torus actions}, Duke Math. J \textbf{62} (1991), no.~2,
  417--451.

\bibitem[EG84]{eisenbud1984linear}
David Eisenbud and Shiro Goto, \emph{Linear free resolutions and minimal
  multiplicity}, Journal of Algebra \textbf{88} (1984), no.~1, 89--133.

\bibitem[EGHP05]{eisenbud2005restricting}
David Eisenbud, Mark Green, Klaus Hulek, and Sorin Popescu, \emph{Restricting
  linear syzygies: algebra and geometry}, Compositio Mathematica \textbf{141}
  (2005), no.~6, 1460--1478.

\bibitem[Eis95]{eisenbud95:_commut_algeb}
David Eisenbud, \emph{Commutative algebra with a view toward algebraic
  geometry}, GTM, vol. 150, Springer Verlag, New York, 1995.

\bibitem[ES09]{eisenbud2009betti}
David Eisenbud and Frank-Olaf Schreyer, \emph{Betti numbers of graded modules
  and cohomology of vector bundles}, Journal of the American Mathematical
  Society \textbf{22} (2009), no.~3, 859--888.

\bibitem[GL86]{green1986projective}
Mark Green and Robert Lazarsfeld, \emph{On the projective normality of complete
  linear series on an algebraic curve}, Inventiones mathematicae \textbf{83}
  (1986), no.~1, 73--90.

\bibitem[GM03]{GM03}
S.~I. Gelfand and Yu.~I. Manin, \emph{Methods of homological algebra}, second
  ed., Springer Monographs in Mathematics, Springer, 2003.

\bibitem[J{\'S}03]{januszkiewicz2003hyperbolic}
Tadeusz Januszkiewicz and Jacek {\'S}wi{\k{a}}tkowski, \emph{Hyperbolic
  {C}oxeter groups of large dimension}, Commentarii Mathematici Helvetici
  \textbf{78} (2003), no.~3, 555--583.

\bibitem[J{\'S}06]{JS-non-positive-curvature}
\bysame, \emph{Simplicial nonpositive curvature}, Pub. Math. IHES \textbf{104}
  (2006), no.~1, 1--85.

\bibitem[J{\'S}10]{januszkiewicz10:_non}
\bysame, \emph{Non-positively curved developments of billiards}, Journal of
  Topology \textbf{3} (2010), no.~1, 63--80.

\bibitem[Knu76]{knuth1976mathematics}
Donald~Ervin Knuth, \emph{Mathematics and computer science: coping with
  finiteness}, Science \textbf{194} (1976), no.~4271, 1235--1242.

\bibitem[Lyu84]{lyubeznik1984local}
Gennady Lyubeznik, \emph{On the local cohomology modules {$H^i_{{\mathfrak
  A}}(R)$} for ideals {${\mathfrak A}$} generated by monomials in an
  {$R$}-sequence}, Complete intersections ({A}cireale, 1983), Lecture Notes in
  Math., vol. 1092, Springer, Berlin, 1984, pp.~214--220.

\bibitem[MM82]{mayr82:_compl_word_probl_commut_semig_polyn_ideal}
Ernst~W. Mayr and Albert~A. Meyer, \emph{The complexity of the word problems
  for commutative semigroups and polynomial ideals}, Advances in Mathematics
  \textbf{46} (1982), no.~3, 305--329.

\bibitem[MP18]{mccullough16:_count}
Jason McCullough and Irena Peeva, \emph{Counterexamples to the
  {E}isenbud--{G}oto regularity conjecture}, Journal of the American
  Mathematical Society \textbf{31} (2018), no.~2, 473--496.

\bibitem[MS05]{miller05:_combin_commut_algeb}
Ezra Miller and Bernd Sturmfels, \emph{Combinatorial commutative algebra}, GTM,
  vol. 227, Springer, Berlin, 2005.

\bibitem[MS13]{mccullough2013bounding}
Jason McCullough and Alexandra Seceleanu, \emph{Bounding projective dimension},
  Commutative algebra, Springer, 2013, pp.~551--576.

\bibitem[MSS17]{MSS16}
Linquan Ma, Karl Schwede, and Kazuma Shimomoto, \emph{Local cohomology of {Du}
  {B}ois singularities and applications to families}, Compositio Mathematica
  \textbf{153} (2017), no.~10, 2147--2170.

\bibitem[Mun84]{munkres1984elements}
James~R Munkres, \emph{Elements of algebraic topology}, vol.~2, Addison-Wesley,
  1984.

\bibitem[NP13]{NP13}
Eran Nevo and Irena Peeva, \emph{{$C_4$}-free edge ideals}, Journal of
  Algebraic Combinatorics \textbf{37} (2013), no.~2, 243 -- 248.

\bibitem[Osa13a]{osajda2013combinatorial}
Damian Osajda, \emph{A combinatorial non-positive curvature {I}: weak
  systolicity}, preprint, arXiv:1305.4661 (2013).

\bibitem[Osa13b]{osajda2010construction}
\bysame, \emph{A construction of hyperbolic {C}oxeter groups}, Comment. Math.
  Helv. \textbf{88} (2013), no.~2, 353--367.

\bibitem[PS09]{Peeva2009open}
Irena Peeva and Mike Stillman, \emph{Open problems on syzygies and {H}ilbert
  functions}, Journal of Commutative Algebra \textbf{1} (2009), no.~1,
  159--195.

\bibitem[S{\etalchar{+}}16]{sage}
W.\thinspace{}A. Stein et~al., \emph{{S}age {M}athematics {S}oftware
  ({V}ersion~7.3)}, The Sage Development Team, 2016,
  \mbox{\url{http://www.sagemath.org}}.

\bibitem[Sch09]{Schwede09}
Karl Schwede, \emph{{$F$}-injective singularities are {D}u {B}ois}, Amer. J.
  Math. \textbf{131} (2009), no.~2, 445--473.

\bibitem[Ull14]{Ull14}
Brooke Ullery, \emph{Designer ideals with high {C}astelnuovo--{M}umford
  regularity}, Math. Res. Lett. \textbf{21} (2014), 1215--1225.

\end{thebibliography}
\end{document}